\documentclass[smallextended,envcountsect]{svjour3}      

\smartqed  
\newcommand{\dom}{\text{dom}}
\newcommand{\epi}{\text{epi}}

\newcommand{\R}{\mathbb{R}}

\newcommand{\vx}{\textbf{x}}
\newcommand{\vy}{\textbf{y}}
\newcommand{\vz}{\textbf{z}}
\newcommand{\vh}{\textbf{h}}

\newcommand{\bD}{\textbf{D}}
\newcommand{\orcid}[1]{\href{https://orcid.org/#1}{\includegraphics[scale=1]{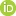}}}

\usepackage{multirow}
\usepackage{subfigure}
\usepackage{graphicx}
\usepackage{amsmath,amsfonts}
\usepackage{amssymb}
\usepackage{cases}
\usepackage{adjustbox}
\usepackage[graphicx]{realboxes}
\usepackage[figuresright]{rotating}
\usepackage{lscape}
\usepackage[colorlinks,linkcolor=blue,anchorcolor=blue,citecolor=blue,urlcolor=blue]{hyperref}
\usepackage[ruled,linesnumbered]{algorithm2e}
\usepackage{listings}
\graphicspath{{figs/}}

\DeclareMathOperator*{\argmin}{argmin}

\newtheorem{assumption}{Assumption}

\begin{document}
\title{A Variable Metric and Nesterov Extrapolated Proximal DCA with Backtracking for A Composite DC Program}
\titlerunning{A variable metric and extrapolated proximal DCA with backtracking}

\author{Yu You \and Yi-Shuai Niu\orcid{0000-0002-9993-3681}}
\authorrunning{Yu You \and Yi-Shuai Niu } 

\institute{
           Yu You \at
           School of Mathematical Sciences, Shanghai Jiao Tong University, Shanghai, China \\
           \email{youyu0828@sjtu.edu.cn}
			 \and
			Yi-Shuai Niu \at
			Department of Applied Mathematics, The Hong Kong Polytechnic University, Hong Kong \\
			\email{yi-shuai.niu@polyu.edu.hk}
}

\date{}

\maketitle

\begin{abstract}
In this paper, we consider a composite difference-of-convex (DC) program, whose objective function is the sum of a smooth convex function with Lipschitz continuous gradient, a proper closed and convex function, and a continuous concave function. This problem has many applications in machine learning and data science. The proximal DCA (pDCA), a special case of the classical DCA, as well as two Nesterov-type extrapolated DCA -- ADCA (Phan et al. IJCAI:1369--1375, 2018) and pDCAe (Wen et al. Comput Optim Appl 69:297--324, 2018) -- can solve this problem. The algorithmic step-sizes of pDCA, pDCAe, and ADCA are fixed and determined by estimating a prior the smoothness parameter of the loss function. However, such an estimate may be hard to obtain or poor in some real-world applications. Motivated by this difficulty, we propose a variable metric and Nesterov extrapolated proximal DCA with backtracking (SPDCAe), which combines the backtracking line search procedure (not necessarily monotone) and the Nesterov's extrapolation for potential acceleration; moreover, the variable metric method is incorporated for better local approximation. Numerical simulations on sparse binary logistic regression and compressed sensing with Poisson noise demonstrate the effectiveness of our proposed method.
\end{abstract}

\section{Introduction}
\label{intro} 
Difference-of-convex (DC) programming has been an active area of research in nonconvex and nonsmooth optimization \cite{hiriart1985generalized,Horst1999,TAO1986249,Pham_dca_theory,DCA30,surfaceDC,de2020abc}. They arise in extensive applications such as sparsity learning \cite{Jackl12,le2015dc}, clustering \cite{an2007new}, molecular conformation \cite{an2003large}, trust region subproblem \cite{Pham_trs,beck2018globally}, portfolio optimization \cite{pham2011efficient,pham2016dc,niu2019higher}, control theory \cite{niu2014dc}, natural language processing \cite{niu2021difference}, image denoising \cite{RInDCA}, mixed-integer optimization \cite{le2001continuous,niu2008dc,niu2010programmation,niu2019parallel} and eigenvalue complementarity problem \cite{le2012dc,niu2013efficient,niu2015solving,niu2019improved}, to name a few; see \cite{DCA30} and the references therein for a comprehensive introduction about DC programming and its applications.

In this paper, we focus on a composite DC program in form of
\begin{equation*}
\begin{split}
\text{min} &\{F(\vx) := f(\vx) + g(\vx)-h(\vx):\vx\in\R^n\}, \\
\end{split}
\label{P}
\tag{$P$}
\end{equation*}
where the following are assumed throughout the manuscript:
\begin{assumption}
	\label{intro:assump1}
	\begin{itemize}
		\item[(a)] $f$ and $g$ are proper convex functions from $\mathbb{R}^n$ to $(-\infty,\infty]$, $h:\mathbb{R}^n\rightarrow \mathbb{R}$ is convex; moreover, $g$ is closed.
		\item[(b)] There exists a nonempty closed and convex set $Y\subseteq \R^n$ such that $\text{dom}(g) \subseteq Y \subseteq \text{int}( \text{dom}(f))$.
		\item[(c)] $f$ is $L$-smooth over $Y$, i.e., the gradient of $f$ is $L$-Lipschitz over $Y$.
		\item[(d)]  $F$ is bounded from below, i.e., there exists scalar $F^*$ such that {\normalfont $F(\vx) \geq F^*$ for all $\vx \in \R^n$}.
	\end{itemize}
\end{assumption}
 Note that we do not require that the loss function $f$ is real-valued over $\R^n$ since the scope of the model can be limited in this way. Next, by setting $h\equiv0$ in \eqref{P} we derive the  convex optimization problem
 \begin{equation*}
 	\begin{split}
 		\text{min} &\{\Phi(\vx) := f(\vx) + g(\vx):\vx\in\R^n\}, \\
 	\end{split}
 	\label{CP}
 	\tag{$P_c$}
 \end{equation*}
for which we additionally assume the following:
\begin{assumption}
	\label{intro:assump2}
	The optimal solution set of \eqref{CP} is nonempty, and its optimal value is denoted by $\Phi^*$.	
\end{assumption} 
  
The classical DCA \cite{Pham_dca_theory,Pham_trs,Lethi_2005,phamdinh2014} is renowned for solving \eqref{P}, which has been introduced by Pham Dinh Tao in 1985 and extensively developed by Le Thi Hoai An and Pham Dinh Tao since 1994. One possible (commonly used) DC decomposition for \eqref{P} takes the following form
\begin{equation}
	\min\{F(\textbf{x}):=G(\textbf{x})-H(\textbf{x})\footnote{The convention $\infty -\infty$ is adopted.}: \textbf{x}\in \R^n\},
	\label{ep}
\end{equation}
where $G(\vx)=L\Vert \vx \Vert^2/2 + g(\vx)$ and $H(\vx) = L\Vert \vx \Vert^2/2 - f(\vx) + h(\vx)$ when $\vx \in Y$ and $\infty$ otherwise. Then DCA applied for this DC decomposition yields the proximal DC algorithm (pDCA) proposed in \cite{DC_Takeda}. For the special case where $Y=\R^n$, the accelerated variant of DCA (ADCA) \cite{nhat2018accelerated}, which incorporated the Nesterov's extrapolation \cite{nesterov1983method,nesterov2013gradient} for potential acceleration, can be applied for solving \eqref{ep}; meanwhile, Wen et al. \cite{wen2018proximal} used the same extrapolation technique for \eqref{P} and proposed their accelerated algorithm pDCAe, which is a variant of pDCA with additional extrapolation term. The major difference between ADCA and pDCAe may be that ADCA needs to determine whether or not to conduct the extrapolation $\vz^k = \vx^k + \beta_k (\vx^k-\vx^{k-1})$  by comparing the functions values of the current extrapolation point $\vz^k$ and the previous $q+1$ iterates $\vx^k,\cdots,\vx^{k-q}$, while there is not such a necessity for pDCAe. Thus, ADCA costs more per iteration than pDCAe. However, ADCA has a wider range of applications. For example, pDCAe will reduce to DCA in the image denoising with nonconvex total variation regularizer~\cite{de2019inertial}, thus lossing the effect of extrapolation; while ADCA still occupies the extrapolation term and performs better than DCA \cite{RInDCA}. Note that the model for both ADCA and pDCAe assumes that $f$ is real-valued over the whole space, which limits the scope of the algorithms for some loss functions with domain not being $\R^n$, such as the generalized Kullback-Leibler divergence in the Poisson noise model \cite{SPIRAL}. Indeed, ADCA and pDCAe can be generalized for addressing this situation by additionally projecting the extrapolation point onto the set $Y$.

Now, let us consider the convex model \eqref{CP}. In some practical applications, the Lipschitz constant of the gradient of $f$ might be hard to estimate. The (monotone) backtracking line search procedure \cite{FISTA} is a common strategy for addressing this issue. However, in this way the step-size is nonincreasing, which may seriously affect the convergence speed if the local curvature of the loss function is relatively large near the initial iterates but small near the tail; see \cite{AdaptiveFISTA} for an illustration of this by the compressed sensing and the binary logistic regression problems. In addition, for those problems with the above property, the (global) Lipschitz constant is generally lager than the local one, thus the constant step-size strategy is not recommended. Taking this into account, \cite{AdaptiveFISTA} proposed a non-monotone backtracking strategy for addressing this kind of problems. Note that model \eqref{P} inherits the above issues in \eqref{CP}.

 In this paper, in term of the success of Nesterov's extrapolation techniques for acceleration, the variable metric method for better local approximation \cite{SFBEM,BONETTINI2021113192}, and the non-monotone backtracking for adaptive step-sizes selection, we incorporate all of these techniques to propose a scaled proximal DC algorithm with extrapolation and backtracking (SPDCAe) for solving problem \eqref{P}. We prove that for suitable selections of the extrapolation parameters and the scaling matrices, if the generated sequence of SPDCAe is bounded, then any limit point of this sequence is a critical point of \eqref{P}. Moreover, we also demonstrate that SPDCAe for the convex case \eqref{CP}, denote as SFISTA, enjoys the optimal $\mathcal{O}(1/k^2)$ convergence rate in function values. This rate of convergence coincides with that of the well-known FISTA algorithm \cite{FISTA} and SFBEM \cite{SFBEM} (a scaling version of FISTA with monotone backtracking). Finally, we point out that our SPDCAe (resp. SFISTA) actually covers pDCAe (resp. SFBEM).

The rest of the paper is organized as follows. Sect. \ref{sec:2} reviews some notations and preliminaries in convex analysis. In Sect. \ref{sec:3}, we describe the DC optimization problem as well as its convex case we are concerned in this paper and present our algorithm SPDCAe and its convex version SFISTA. Sect. \ref{sec:4} focuses on establishing the subsequential convergence of SPDCAe and the optimal convergence rate of SFISTA.  Numerical results on problems of sparse binary logistic regression and compressed sensing with Poisson noise are summarized in Sect. \ref{sec:5}.

\section{Notations and preliminaries}
\label{sec:2}
Let $\R^n$ denote the $n$-dimensional column vector space endowed with  canonical inner product $\langle \cdot, \cdot \rangle$ and its induced norm $\Vert \cdot \Vert$. For an extended real-valued function $p:\R^n\rightarrow (-\infty,\infty]$, the set  $$\dom(p) := \{\textbf{x}\in \R^n:p(\textbf{x}) < \infty \}$$ denotes its effective domain. If $\dom (p) \neq \emptyset$ and $p$ does not attain the value $-\infty$, then $p$ is called a proper function. The notation $\text{int}(\dom(p))$ denotes the interior set of $\dom(f)$, and 
$$\epi (p):=\{(\textbf{x},t)\in \R^n \times \R:p(\textbf{x})\leq t\}$$ denotes the epigraph of $p$. We have the definition that $p$ is closed (or convex) if $\epi (p)$ is closed (or convex).


Let $p:\R^n\rightarrow (-\infty,\infty]$ be a proper function and $\vx \in \text{int}(\dom(p))$. Then $p$ is said to be differentiable at $\vx$ if there exists $\textbf{g}\in\R^n$ such that 
$$\lim\limits_{\textbf{d}\rightarrow 0} \frac{p(\vx+\textbf{d})-p(\vx)-\langle \textbf{g},\textbf{d}\rangle}{\Vert \textbf{d}\Vert} = 0.$$ The unique vector $\textbf{g}$ as above is called the gradient of $p$ at $\vx$, and is denoted as $\nabla p(\vx)$. Finally, we say that $p$ is $L$-smooth over $Y$ if $p$ is differentiable over $Y$ and $\Vert \nabla p(\vx)-\nabla p(\vy)\Vert \leq L \Vert \vx-\vy\Vert$ for any $\vx$ and $\vy$. In addition, if $Y$ is convex, then we have the next descent lemma.
\begin{lemma}[see e.g., \cite{Beck_1order}]
Let $p:\R^n\rightarrow (-\infty,\infty]$ be an $L$-smooth function ($L\geq 0$) over a given convex set $Y$. Then for any {\normalfont $\vx, \vy \in Y$}, 
\begin{equation}
	{\normalfont
	p(\vy) \leq p(\vx) + \langle \nabla p(\vx),\vy-\vx\rangle + \frac{L}{2}\Vert \vx-\vy\Vert^2.
	}
\end{equation} 
\end{lemma}

Now, let $p:\R^n\rightarrow (-\infty,\infty]$ be a proper closed and convex function. A vector $\textbf{g}\in \R^n$ is called a subgradient of $p$ at $\vx$ if $$p(\vy) \geq p(\vx) +\langle \textbf{g},\vy-\vx\rangle ,\; \forall  \vy\in \R^n,$$ and the subdifferential of $g$ at $\vx$ is defined by
$$
\partial p(\vx)=\{\textbf{g}\in \R^n:p(\vy) \geq p(\vx) +\langle \textbf{g},\vy-\vx\rangle ,\; \forall \vy\in \R^n\}.
$$
The proximal mapping of $p$ is the operator defined by 
\begin{equation*}
	\text{Prox}_{p}(\vx) = \argmin_{\textbf{u}\in\R^n} \left\{p(\textbf{u})+\frac{1}{2}\Vert \textbf{u}-\vx\Vert^2 \right\} ,\; \forall  \vx\in \R^n.
\end{equation*}
Next, we summarize some properties about the proximal mapping. 
\begin{lemma}[see e.g., \cite{Beck_1order}]
	\label{lem:2.2beck}
	Let $p:\R^n\rightarrow (-\infty,\infty]$ be a proper closed and convex function. Then
	\begin{itemize}
		\item ${\normalfont \text{Prox}_{p}(\vx)}$ is a singleton for any ${\normalfont \vx} \in \R^n$.
		\item ${\normalfont \text{Prox}_{p}}$ is nonexpansive, i.e., for any ${\normalfont \vx}$ and ${\normalfont \vy}$ it holds that
		$${\normalfont \Vert \text{Prox}_{p}(\vx)-\text{Prox}_{p}(\vy)\Vert  \leq \Vert \vx -\vy\Vert}.$$ 
		\item If ${\normalfont \vz = \text{Prox}_{p}(\vx)}$, then ${\normalfont \vx-\vz \in \partial p(\vz)}$.
	\end{itemize}
\end{lemma}
Note that if $p$ is the indicator function $\delta_{Y}$ of a nonempty and closed convex set $Y$, i.e., $\delta_{Y}(\vx)$ is equal to $0$ if $\vx\in Y$ and $\infty$ otherwise, then $\text{Prox}_{\delta_{Y}}$ is the projection operator over Y, denoted as ${\rm \Pi}_Y$. The notations $\text{Prox}_{p}^{\textbf{Q}}$  and ${\rm \Pi}_Y^{\textbf{Q}}$ ($\textbf{Q}$ is a positive definite matrix) are adopted, the former actually means
\begin{equation*}
	\text{Prox}_{p}^{\textbf{Q}}(\vx) = \argmin_{\textbf{u}\in\R^n} \left\{p(\textbf{u})+\frac{1}{2}\Vert \textbf{u}-\vx\Vert_{\textbf{Q}}^2 \right\} ,\; \forall \vx\in \R^n.
\end{equation*}
In this case, $\vz = \text{Prox}_{p}^{\textbf{Q}}(\vx)$ implies that $\vx -\vz \in \textbf{Q}^{-1}\partial p(\vz)$.

For two vectors $\vx, \vy \in \R^n$, their Hardmard product,  denoted as $\vx \odot \vy$, is the vector comprising the component-wise products: $\vx \odot \vy = (x_iy_i)_{i=1}^n$. The nonnegative orthant is denoted by $\R_{+}^n$. Next, $\text{diag}(\vx)$ denotes the $n\times n$ diagonal matrix with its diagonal vector being $\vx$. The set of all $n\times n$ symmetric matrices is denoted as $\mathbb{S}^n$; moreover, $\mathbb{S}_{+}^n$ (or $\mathbb{S}_{++}^n$) means the set of all $n\times n$ positive semidefinite (or positive definite) matrices. For $\bD_1, \bD_2 \in \mathbb{S}^n$, $\bD_1  \preceq \bD_2$ means that $\bD_2-\bD_1\in \mathbb{S}_{+}^n$. For $\mu >0$, $\bD_{\mu}$ denotes the set $\{\bD\in\mathbb{S}^n:\mu \textbf{I} \preceq \bD\}$, where $\textbf{I}$ is the identity matrix.

Finally, we present a technical lemma, adapted from \cite[Proposition A.31]{bertsekas2015convex}, which will be utilized for establishing the convergence property of our proposed method SPDCAe for \eqref{P} and its convex version SFISTA for \eqref{CP}.
\begin{lemma}[\cite{bertsekas2015convex}]
	\label{ine1}
	Let $\{a_k\}, \{\zeta_k\}$, $\{\xi_k\}$, and $\{\gamma_k\}$ be nonnegative sequences of real numbers such that 
	\begin{equation*}
		a_{k+1} \leq (1+\zeta_k)a_k - \xi_k + \gamma_k, \quad  k = 1,2,\cdots,
	\end{equation*}
and 
\[
\sum_{k=1}^{\infty}\zeta_k < \infty, \quad \sum_{k=1}^{\infty}\gamma_k < \infty.
\]
Then $\{a_k\}$ converges to a finite value and $ \{\xi_k\}$ converges to 0.	
\end{lemma}


\section{Problem formulation and the proposed algorithms} \label{sec:3}
In this section, we first introduce the DC optimization problem we focus on in this paper and present our proposed algorithm SPDCAe for \eqref{P}, which incorporates the Nesterov's extrapolation, the backtracking line search (monotone and non-monotone), and the variable metric method. Next, we describe in detail the algorithmic procedure of SFISTA, which is the convex version of SPDCAe with non-monotone backtracking for \eqref{CP}.    

In this paper, we specifically focus on the  composite DC program
\begin{equation*}
	\begin{split}
		\text{min} &\{F(\vx) := f(\vx) + g(\vx)-h(\vx):\vx\in\R^n\}. \\
	\end{split}
	\label{P}
	\tag{$P$}
\end{equation*}
Some assumptions for the involved functions are summarized in Assumption \ref{intro:assump1}.  Note that we do not follow the assumption of the two  accelerated algorithms ADCA and pDCAe developed respectively in \cite{nhat2018accelerated,wen2018proximal} that $f$ is real-valued over the whole space since the scope of the model can be limited in this way. For example, the loss function in the Poisson noise model \cite{SPIRAL} is the generalized Kullback-Leibler divergence, whose effective domain is larger than the non-negative orthant but not equal to the whole space. Indeed, their algorithms can be generalized for this situation by additionally projecting the extrapolation point onto the set $Y$ at each iteration. 
\begin{algorithm}[ht!]
	\caption{SPDCAe for \eqref{P}}
	\label{algo:ASpDCAe}
	\KwIn{$\eta >1, \underline{L}>0$. Set $\vx^0 = \vx^{-1} \in Y \subseteq \R^n$. }
	
	\For{$k=1, 2, 3, \cdots$}
	{   
		compute any $\vh^{k-1} \in \partial h(\vx^{k-1})$ and pick $L_k \geq \underline{L}$;
		
		set $i=0$;
		
		\Repeat
		{
			{\normalfont$f(\vx^k) \leq f(\vy^k) + \langle \nabla f(\vy^k),\vx^k-\vy^k\rangle + \frac{1}{2 t_k}\Vert \vx^k-\vy^k\Vert_{\textbf{D}_k}^2$}
		}
		{
			(1) update $L_k=\eta^i L_k$ and set $t_k = 1/L_k$;	\\
			(2) pick $\beta_k \in [0,1)$ and $\textbf{D}_k \in \mathbb{S}_{++}^n$, calculate $\vy^k =  {\rm \Pi}_Y^{\textbf{D}_k}\left( \vx^{k-1}+ \beta_k(\vx^{k-1}-\vx^{k-2}) \right)$;\\
			(3)	compute  $\vx^k = \text{Prox}_{ t_k g}^{\textbf{D}_k}(\vy^k- t_k  \textbf{D}_k^{-1} [\nabla f(\vy^k)-\vh^{k-1}])$;\\
			(4) update $i=i+1$;
		} 
	}	
\end{algorithm}
We describe in Algorithm \ref{algo:ASpDCAe} our proposed algorithm SPDCAe for solving problem \eqref{P}. The parameters $\beta_k$, $L_k$, and $\bD_k$ remain to be specified. Later, we will prove that with suitable selections of parameters, any limit point of the sequence generated by SPDCAe, denoted as $\vx^*$, is a critical point of problem \eqref{P}, i.e., 
$$\textbf{0} \in \nabla f(\vx^*) + \partial g(\vx^*) - \partial h(\vx^*).$$ At the moment, we just assume the following for the scaling matrices $\{\bD_k\}$ and defer the specified selection to the numerical part in Sect. \ref{sec:5}.
\begin{assumption}
	\label{assumDk}
	The scaling matrices $\{\bD_k\}$ satisfy that $$\{\bD_k\} \subseteq \bD_{\mu},\quad \bD_{k+1} \preceq (1+\eta_k)\bD_k, \quad k=1,2,\cdots,$$  where $\mu>0$ and $\{\eta_k\}$ is a nonnegative sequence such that $\sum_{k=1}^{\infty} \eta_k < \infty$.
\end{assumption}
\begin{remark}
\label{remark:3.1}
Assumption \ref{assumDk} implies that there exists $c>0$ such that 
$$
\mu\textbf{I} \preceq \bD_{k} \preceq c \bD_1,\quad k=1,2,\cdots,
$$
where $c = \prod_{k=1}^{\infty}(1+\eta_k)$.
\end{remark}
To determine $\beta_k$ and $L_k$, we distinguish the next two backtracking strategies: 

\begin{itemize}
	\item[$\bullet$] \textbf{Non-monotone backtracking} Here, we use the heuristic method suggested in \cite{AdaptiveFISTA} to determine $\{L_k\}$. Specifically, let $T_1$ be a positive integer, the initial guess $L_k$ (as in line 2 of Algorithm \ref{algo:ASpDCAe}) is set as the previously returned $L_{k-1}$ if $k$ is not divisible by $T_1$ and as $\rho L_k$ otherwise, where $\rho\in(0,1)$. Next, two common methods for updating  $\{\beta_k\}$ is as follows: 
\begin{itemize}
	\item \textbf{Contract method} sets $\beta_1 = 0$ and $\beta_k = \frac{\theta_{k-1}-1}{\theta_k} \delta$ for $k\geq 2$, where $\delta \in (0,1)$ is a contractive factor (generally close to 1) and $\{\theta_k\}_{k\geq 1}$ is the sequence recursively defined by 
		\begin{equation}
		\label{updation}
	\theta_1 = 1,\quad \theta_l = \frac{1+\sqrt{1+4\theta_{l-1}^2  L_l/L_{l-1}}}{2}, \quad l=2,3,\cdots.    	
		\end{equation} 
	\item \textbf{Restart method} updates $\{\beta_k\}$ via the \emph{fixed restart scheme} or \emph{the fixed and adaptive restart scheme} proposed in \cite{restartFISTA}. Specifically, let $T_2$ be a positive integer. For the fixed restart scheme, one starts to update $\{\theta_k\}$ by \eqref{updation} and obtain $\beta_1 = 0$ and $\beta_k  = \frac{\theta_{k-1}-1}{\theta_k}$ for $k\geq 2$, then reset $\theta_{k+1} = 1$ if $T_2$ is divisible by $k$. For the fixed and adaptive restart scheme, the restart condition shall be modified as either $T_2$ is divisible by $k$ or $\langle \vx^k-\vx^{k-1}, \vy^k-\vx^k\rangle >0$.
\end{itemize}   

\item[$\bullet$] \textbf{Monotone backtracking}
In this situation, the initial trail $L_k$ is set as the previously returned $L_{k-1}$. Then, about the update of $\{\beta_k\}$, the only difference from the non-monotone one is that \eqref{updation} should be modified as $$\theta_1 = 1, \quad \theta_l = (1+\sqrt{1+4\theta_{l-1}^2})/2,\quad l=2,3,\cdots.$$
\end{itemize}

\begin{remark}
	\label{remark3.2}
	\begin{enumerate}
		\item The restart schemes as described above were also used in \cite{wen2018proximal} for their proposed  algorithm pDCAe. Moreover, our SPDCAe with monotone backtracking is exactly pDCAe if the domain of $f$ is the whole space, $L_k$ is selected as the smoothness parameter of $f$, and the scaling matrix $\textbf{D}_k$ is set as the identity matrix at each iteration.
		\item For SPDCAe with monotone backtracking, $\vx^{k-1}+ \beta_k(\vx^{k-1}-\vx^{k-2})$ can be computed out of the inner loop.
		\item The inner loop in Algorithm \ref{algo:ASpDCAe} can stop in finite step \cite{SFBEM}.
		\item For the monotone backtracking, $L_k \geq \underline{L}$ can be omitted, besides, the boundedness of $\{\bD_{k}\}$ as in Remark \ref{remark:3.1} can imply that $0<\liminf L_k \leq \limsup L_k < \infty$; while for the non-monotone one, $L_k \geq \underline{L}$ guarantees that $\liminf L_k >0 $, and thus $0<\liminf L_k \leq \limsup L_k < \infty$. 
	\end{enumerate}
\end{remark}

Next, we describe in Algorithm \ref{algo:ASFISTA} the algorithmic procedure of SFISTA with non-monotone backtracking line search (the convex version of SPDCAe with non-monotone backtracking).  
\begin{algorithm}[ht!]
	\caption{SFISTA with non-monotone backtracking for $\eqref{CP}$}
	\label{algo:ASFISTA}
	\KwIn{$\theta_0 = 1$, $t_0 = 0$, $\eta > 1$. Set $\vx^0 = \vx^{-1} \in Y \subseteq \R^n$.}
	
	\For{$k=1, 2, 3, \cdots$}
	{   		
		pick $L_k>0$;
		
		set $i = 0$;
		
		\Repeat
		{
			{\normalfont$f(\vx^k) \leq f(\vy^k) + \langle \nabla f(\vy^k),\vx^k-\vy^k\rangle + \frac{1}{2 t_k}\Vert \vx^k-\vy^k\Vert_{\textbf{D}_k}^2$}
		}
		{
			(1) update $L_k$ by $\eta^i L_k$ and set $t_k = 1/L_k$;\\
			(2) compute $\theta_k = \frac{1+\sqrt{1+4\theta_{k-1}^2  t_{k-1}/t_{k}}}{2}$;	\\
			(3) pick $\textbf{D}_k \in \mathbb{S}_{++}^n$, calculate $\vy^k =  {\rm \Pi}_Y^{\textbf{D}_k}\left( \vx^{k-1}+\frac{\theta_{k-1}-1}{\theta_k}(\vx^{k-1}-\vx^{k-2}) \right)$;\\
			(4)	compute  $\vx^k = \text{Prox}_{t_k g}^{\textbf{D}_k}(\vy^k- t_k  \textbf{D}_k^{-1} \nabla f(\vy^k) )$;\\
			(5) update $i$ by $i+1$.
			
		}
	}	
\end{algorithm}
Note that SFISTA with monotone backtracking is exactly SFBEM proposed in \cite{SFBEM} and thus we omit to present it here. In \cite{AdaptiveFISTA},  Scheinberg et al. showed that for some practical applications, such as the compressed sensing and the logistic regression problems, the local Lipschitz constant of the gradient of loss function is generally smaller than the global one, especially when approaching the tail. Thus, the monotone backtracking will lead to slow convergence near the solution set, while the non-monotone backtracking helps to overcome this issue to some extent. Later, we will show that SFISTA is equipped with the optimal convergence rate $\Phi(\vx^k)-\Phi^*\leq \mathcal{O}(1/k^2)$, which coincides with that of SFBEM and the classical FISTA \cite{FISTA}.

%
%
%
%

\section{Convergence analysis}
\label{sec:4}
In this section, we will investigate the convergence property of  SPDCAe for \eqref{P} and its convex version SFISTA for \eqref{CP}. For the former, we demonstrate the subsequential convergence to a critical point; while for the latter, the optimal convergence rate $\Phi(\vx^k)-\Phi^*\leq \mathcal{O}(1/k^2)$ is established.

The convergence analyses of both SPDCAe and SFISTA are based on the following key inequality. 
\begin{proposition}
	\label{conv_pro}
	Suppose that $f$, $g$, and $h$ satisfy properties (a), (b), and (c) of Assumption \ref{algo:ASpDCAe}. For any ${\normalfont \vx \in \R^n, \vh \in \partial h(\vx), \vy \in Y}$, $\bD \in \mathbb{S}_{++}^n$, and $t>0$ satisfying
	\begin{equation}
	\label{prop_in1}
	{\normalfont f(\bar{\vy}) \leq f(\vy) + \langle \nabla f(\vy),\bar{\vy}-\vy\rangle + \frac{1}{2 t}\Vert \bar{\vy}-\vy\Vert_{\textbf{D}}^2}, 	
	\end{equation}
	where ${\normalfont \bar{\vy}:= \text{Prox}_{t g}^{\textbf{D}}(\vy- t  \textbf{D}^{-1} [\nabla f(\vy) -\vh ] )}$, it holds that 
	$${\normalfont F(\bar{\vy}) \leq F(\vx) + \frac{1}{2t}\Vert \vx -\vy \Vert_{\textbf{D}}^2 -\frac{1}{2t}\Vert \vx-\bar{\vy} \Vert_{\textbf{D}}^2
	}.$$
\end{proposition}
\begin{proof}
	It follows from $\bar{\vy} = \text{Prox}_{t g}^{\bD}(\vy- t  \bD^{-1} [\nabla f(\vy)-\vh])$ that $\vy- t  \bD^{-1} [\nabla f(\vy)-\vh] - \bar{\vy} \in t\textbf{D}^{-1}\partial g(\bar{\vy})$ (term 3 of Lemma \ref{lem:2.2beck}), which has the equivalent form
		\begin{equation*}
		\label{next001}
		\bD(\vy-\bar{\vy})-t[\nabla f(\vy)-\vh]\in t \partial g(\bar{\vy}).
	\end{equation*}
Then, it holds that 
\begin{equation}
\label{eq:0002}
tg(\vx) \geq  tg(\bar{\vy}) + \langle \vx - \bar{\vy},\textbf{D}(\vy-\bar{\vy}) -t[\nabla f(\vy)-\vh]\rangle.	
\end{equation}
Thus, we have 
\begin{align*}
	&F(\bar{\vy})  = f(\bar{\vy}) + g(\bar{\vy}) - h(\bar{\vy}) \\
	\overset{\eqref{prop_in1},\eqref{eq:0002}}{\leq} & f(\vy) + \langle \nabla f(\vy),\bar{\vy}-\vy\rangle + \frac{1}{2 t}\Vert \bar{\vy}-\vy\Vert_{\textbf{D}}^2 \\
	&+g(\vx) -  \frac{1}{t} \langle \vx-\bar{\vy},\bD(\vy-\bar{\vy})-t[\nabla f(\vy)-\vh] \rangle -h(\bar{\vy}) \\
	=& f(\vy) + g(\vx) + \langle  \nabla f(\vy),\vx-\vy \rangle + \frac{1}{t}\langle \vx-\bar{\vy},\bD(\bar{\vy}-\vy)\rangle +\frac{1}{2 t}\Vert \bar{\vy}-\vy\Vert_{\bD}^2 \\
	& -\langle \vx - \bar{\vy},\vh\rangle - h(\bar{\vy}) \\
	\leq& F(\vx) +\frac{1}{t}\langle \vx-\bar{\vy},\bD(\bar{\vy}-\vy)\rangle +\frac{1}{2 t}\Vert \bar{\vy}-\vy\Vert_{\bD}^2  -\langle \vx - \bar{\vy},\vh\rangle - h(\bar{\vy}) +h(\vx) \\
	\leq& F(\vx) + \frac{1}{t}\langle \vx-\bar{\vy},\bD(\bar{\vy}-\vy)\rangle +\frac{1}{2 t}\Vert \bar{\vy}-\vy\Vert_{\bD}^2 \\
	=& F(\vx) + \frac{1}{2t}\Vert \vx -\vy \Vert_{\textbf{D}}^2 -\frac{1}{2t}\Vert \vx-\bar{\vy} \Vert_{\textbf{D}}^2,
\end{align*}
where the second inequality follows by $f(\vy) +\langle f(\vy),\vx-\vy\rangle \leq f(\vx)$, the last inequality holds since $\vh \in \partial h(\vx)$ and thus $h(\bar{\vy}) \geq h(\vx) + \langle \bar{\vy}-\vx,\vh\rangle$.  \qed
	
\end{proof}
\subsection{Subsequential convergence of SPDCAe}
Now, we start to establish the subsequential convergence of SPDCAe with non-monotone backtracking and the update rule for $\{\beta_k\}$ as introduced in Sect. \ref{sec:3}. The proof for SPDCAe with monotone backtracking follows similarly and thus omitted.
\begin{theorem}
	Suppose that Assumption \ref{intro:assump1} and Assumption \ref{assumDk} hold. Let {\normalfont $\{\vx_k\}$} be the sequence generated by Algorithm \ref{algo:ASpDCAe}. Then, the following statements hold: 
	\begin{enumerate}
		\item[(i)] The sequence ${\normalfont \{F(\vx^k)\}}$ converges and {\normalfont$\lim\limits_{k\rightarrow \infty}{\Vert \textbf{x}^{k}-\textbf{x}^{k-1}\Vert  = 0}$};
		\item[(ii)] If {\normalfont$\{\textbf{x}^k\}$} is bounded, then any limit point of {\normalfont$\{\textbf{x}^k\}$} is a critical point of \eqref{P}. 
	\end{enumerate}
\end{theorem}

\begin{proof}
	$(i)$ Let $k\geq 1$. Replace $\vx$, $\vh$, $\vy$, $\bar{\vy}$, $\bD$, $t$ in Proposition \ref{conv_pro} by $\vx^{k-1}$, $\vh^{k-1}$, $\vy^k$, $\vx^k$, $\bD_{k}$, $t_k$, respectively. Then, we obtain
	\begin{equation}
		\label{them_ine_01}
		F(\vx^k) \leq F(\vx^{k-1}) + \frac{1}{2t_k}\Vert \vx^{k-1}-\vy^k\Vert_{\textbf{D}_k}^2 - \frac{1}{2t_k}\Vert \vx^k-\vx^{k-1}\Vert_{\textbf{D}_k}^2.
	\end{equation}
	Note that $\vy^k =  {\rm \Pi}_Y^{\textbf{D}_k}\left( \vx^{k-1}+\beta_k (\vx^{k-1}-\vx^{k-2}) \right)$ and $\vx^{k-1}\in Y$, thus we have 
	\begin{equation}
		\label{next2}
		\Vert \vx^{k-1}-\vy^k \Vert_{\bD_k}^2 \leq \beta_k^2\Vert \vx^{k-1}-\vx^{k-2}\Vert_{\bD_k}^2.
	\end{equation}
	 Plunge \eqref{next2} into \eqref{them_ine_01} and then  consider the update rule for $\{\beta_k\}$, it is easy to check that for $k\geq 2$,
	\begin{equation*}
		F(\vx^k) \leq F(\vx^{k-1}) + 
		\frac{\alpha_k}{2t_{k-1}}
		\Vert \vx^{k-1}-\vx^{k-2}\Vert_{\bD_k}^2 - \frac{1}{2t_k}\Vert \vx^k-\vx^{k-1}\Vert_{\bD_k}^2,
	\end{equation*}
	where $\alpha_k = \delta^2(1-1/\theta_k)(1-1/\theta_{k-1})^2$ (resp. $\alpha_k = (1-1/\theta_k)(1-1/\theta_{k-1})^2$) for the contract method (resp. restart method) for updating $\{\beta_k\}$  as introduced in Sect. \ref{sec:3}.  Clearly, for either of the selection methods, it holds that $\{\alpha_k\} \subseteq [0,1)$ with $\sup_k \alpha_k < 1$. Next, denote $\widetilde{F}(\vx^{k}) = F(\vx^k)-F^* + \frac{1}{2t_k}\Vert \vx^k-\vx^{k-1}\Vert_{\bD_k}^2$, we have
	\begin{align*}
		&\widetilde{F}(\vx^{k}) = F(\vx^k)-F^* + \frac{1}{2t_k}\Vert \vx^k-\vx^{k-1}\Vert_{\bD_k}^2 \\
		\leq& F(\vx^{k-1}) - F^*+ \frac{1}{2 t_{k-1}}\Vert \vx^{k-1}-\vx^{k-2}\Vert_{\bD_k}^2 
		- \frac{1-\alpha_k}{2 t_{k-1}}	\Vert \vx^{k-1}-\vx^{k-2}\Vert_{\bD_k}^2 \\
		\leq&  F(\vx^{k-1}) - F^*+ \frac{1}{2t_{k-1}} \Vert \vx^{k-1}-\vx^{k-2}\Vert_{(1+\eta_{k-1})\bD_{k-1}}^2 - \frac{1-\alpha_k}{2 t_{k-1}}	\Vert \vx^{k-1}-\vx^{k-2}\Vert_{\bD_k}^2 \\
		\leq & (1+\eta_{k-1})\widetilde{F}(\vx^{k-1})
		- \frac{1-\alpha_k}{2t_{k-1}}	\Vert \vx^{k-1}-\vx^{k-2}\Vert_{\bD_k}^2,
	\end{align*}
where the second inequality follows by Assumption \ref{assumDk} for $\{\bD_{k}\}$.
	Based on the facts that $\sup_k \alpha_k < 1$ and $\{t_k\}$ is  bounded from above (term 4 of Remark \ref{remark3.2}), it follows that the sequence $\{\frac{1-\alpha_k}{2t_{k-1}}\}_{k\geq 2}$ has positive limit inferior.
	Then invoking Lemma  \ref{ine1}, it holds that
	\begin{equation}
		\label{final}
	\lim\limits_{k\rightarrow \infty} \Vert \vx^{k-1}-\vx^{k-2}\Vert_{\bD_k}^2 = 0,	
	\end{equation}
and $\{\widetilde{F}(\vx^k)\}$ converges to a finite value. Clearly, we have $\{F(\vx^k)\}$ converges to a finite value. Besides, $\{\bD_k\}$ satisfies that 
	$ \mu \textbf{I} \preceq \bD_k$ for $k\geq 1$.
	Thus, $\Vert \vx^{k-1}-\vx^{k-2}\Vert_{\bD_k}^2 \geq \mu \Vert \vx^{k-1}-\vx^{k-2}\Vert^2$,  then \eqref{final} implies $\lim\limits_{k\rightarrow \infty} \Vert \vx^{k-1}-\vx^{k-2}\Vert = 0$.\\
	$(ii)$ 
	Take into account that  {\normalfont$\lim\limits_{k\rightarrow \infty}{\Vert \textbf{x}^{k}-\textbf{x}^{k-1}\Vert  = 0}$}  and $ \mu \textbf{I} \preceq \bD_k$ for all $k$, we obtain from \eqref{next2} that {\normalfont$\lim\limits_{k\rightarrow \infty}{\Vert \textbf{x}^{k}-\textbf{y}^{k}\Vert  = 0}.$}  Let $\bar{\textbf{x}}$ be any limit point of $\{\textbf{x}^k\}$, then there exists a convergent subsequence such that $\lim\limits_{i\rightarrow \infty}\textbf{x}^{k_i} = \lim\limits_{i\rightarrow \infty}\textbf{y}^{k_i} = \bar{\textbf{x}}$.  Moreover, because the sequence $\{\vh^{k_i-1}\}$ is bounded in $\R^n$, we can assume without loss of generality that $\{\vh^{k_i-1}\}$ is convergent with limit $\bar{\textbf{z}}$. Then, it follows from $\vh^{k_i-1}\in \partial h(\vx^{k_i-1})$ for all $i$, $\vh^{k_i-1} \rightarrow \bar{\textbf{z}}$, $\vx^{k_i-1}\rightarrow \bar{\textbf{x}}$, and the closedness of the graph $\partial h$ that $\bar{\textbf{z}} \in \partial h(\bar{\textbf{x}})$ \cite[Theorem 24.4]{rockafellar1970convex} . On the other hand,
	$\vx^k = \text{Prox}_{t_k g}^{\bD_k}(\vy^k-t_k \bD_k^{-1}[\nabla f(\vy^k)-\vh^{k-1}])$ implies that 
	\begin{equation}
		\label{next3}
		\frac{1}{t_{k_i}}\bD_{k_i}(\vy^{k_i}-\vx^{k_i})-[\nabla f(\vy^{k_i})-\vh^{k_i-1}]\in  \partial g(\vx^{k_i}).
	\end{equation}
Consider also $\{\bD_{k}\}$ bounded from above (Remark \ref{remark:3.1}), thus the left-hand side of \eqref{next3} converges to $-\nabla f(\bar{\textbf{x}})+\bar{\textbf{z}}$. Combining this with the facts that $\{\vx_{k_i}\}\rightarrow \bar{\vx}$ and the closedness of the graph $\partial g$, we have $$-\nabla f(\bar{\textbf{x}})+\bar{\textbf{z}} \in \partial g(\bar{\vx}),$$
thus $\textbf{0} \in \nabla f(\bar{\textbf{x}})  +\partial g(\bar{\vx}) -\bar{\textbf{z}}$. Then we conclude that $\bar{\vx}$ is a critical point of \eqref{P} since $\bar{\textbf{z}} \in \partial h(\bar{\textbf{x}})$.
\qed
\end{proof}

\subsection{Optimal convergence rate $\mathcal{O}(1/k^2)$ of SFISTA}
\label{sub:3.2} 
In this part, we start to establish the optimal convergence rate of SFISTA with non-monotone backtracking (Algorithm \ref{algo:ASFISTA}) for \eqref{CP}.  

\begin{theorem}
	Suppose that Assumption \ref{intro:assump2} holds. Let {\normalfont $\{\vx_k\}$} be the sequence generated by Algorithm \ref{algo:ASFISTA}. Then, for any optimal solution  {\normalfont$\vx^*$} of \eqref{CP}, the following statements hold: 
	\begin{enumerate}
		\item[(i)] {\normalfont For $k\geq 0$, denote $\textbf{v}^k = \vx^{k-1} +\theta_k(\vx^k-\vx^{k-1})$.} Then, for $k\geq 1$ we have
		{
		\normalfont 
		\[
		t_k \theta_k^2(\Phi(\vx^k)-\Phi^*) + \frac{1}{2}\Vert \vx^*-\textbf{v}^k\Vert_{\bD_k}^2 \leq t_{k-1}\theta_{k-1}^2(\Phi(\vx^{k-1})-\Phi^*) + \frac{1}{2}\Vert \vx^*-\textbf{v}^{k-1}\Vert_{\bD_k}^2.
		\]
		}
		\item[(ii)] If $\{\bD_k\}$ satisfies Assumption \ref{assumDk}, then there exists $C>0$ such that for $k\geq 2$
		{
		\normalfont
		\[
		\Phi(\vx^k) - \Phi^* \leq \frac{C}{(k+1)^2}.
		\]	
		}
	\end{enumerate}
	
\end{theorem}
\begin{proof}
	$(i)$ Let $k\geq 1$. Note that in this situation $h\equiv 0$ and $\partial h(\vx) = \{\textbf{0}\}$. Replace $\vh$, $\vy$, $\bar{\vy}$, $\bD$, $t$ in Proposition \ref{conv_pro} by $\textbf{0}$, $\vy^k$, $\vx^k$, $\bD_{k}$, $t_k$, respectively. Then, we obtain 
	\begin{align}
		\label{cvx_1}
		\Phi(\vx^k) &\leq \Phi(\vx) + \frac{1}{2t_k}\Vert \vx -\vy^k \Vert_{\textbf{D}_k}^2 -\frac{1}{2t_k}\Vert \vx-\vx^k \Vert_{\textbf{D}_k}^2 \nonumber \\  
		&=\Phi(\vx)+ \frac{1}{t_k}\langle \vx-\vx^k,\bD_k(\vx^k-\vy^k)\rangle +\frac{1}{2t_k}\Vert \vx^k-\vy^k\Vert_{\bD_k}^2.
	\end{align}
Set $\vx=\vx^{k-1}$ and $\vx=\vx^*$ respectively in \eqref{cvx_1}, it follows that
	\begin{align*}
		\Phi(\vx^k) \leq & \Phi(\vx^{k-1})  + \frac{1}{t_k}\langle \vx^{k-1}-\vx^k,\bD_k(\vx^k-\vy^k)\rangle +\frac{1}{2t_k}\Vert \vx^k-\vy^k\Vert_{\bD_k}^2, \\
		\Phi(\vx^k) \leq & \Phi^* +  \frac{1}{t_k}\langle \vx^*-\vx^k,\bD_k(\vx^k-\vy^k)\rangle +\frac{1}{2t_k}\Vert \vx^k-\vy^k\Vert_{\bD_k}^2.
	\end{align*} 
	Multiply the above two inequalities by $1-\theta_k^{-1}$ and $\theta_k^{-1}$ respectively and next add them together, then we obtain 
	\begin{align}
		\label{convex_key}
		\Phi(\vx^k)-\Phi^* \leq& (1-\theta_k^{-1})(\Phi(\vx^{k-1})-\Phi^*) + \frac{1}{2t_k}\Vert \vx^k-\vy^k\Vert_{\bD_k}^2 \nonumber \\
		&+ \frac{1}{t_k}\langle (1-\theta_k^{-1})\vx^{k-1}-\vx^k+\theta_k^{-1} \vx^*,\bD_k(\vx^k-\vy^k) \rangle \nonumber \\
		=& (1-\theta_k^{-1})(\Phi(\vx^{k-1})-\Phi^*) +\frac{1}{2t_k}\Vert (1-\theta_k^{-1})\vx^{k-1} + \theta_k^{-1} \vx^*- \vy^k\Vert_{\bD_k}^2 \nonumber \\
		&- \frac{1}{2t_k} \Vert (1-\theta_k^{-1})\vx^{k-1}+\theta_k^{-1} \vx^*- \vx^k\Vert_{\bD_k}^2.
	\end{align}
	 Note that 
	\begin{align*}
		\vy^k &= {\rm \Pi}_Y^{\textbf{D}_k} \left( \vx^{k-1}+\frac{\theta_{k-1}-1}{\theta_k}(\vx^{k-1}-\vx^{k-2}) \right) \\
		&= {\rm \Pi}_Y^{\textbf{D}_k} \left( (1-\theta_k^{-1})\vx^{k-1} + \theta_k^{-1} \textbf{v}^{k-1}\right) 
	\end{align*}
	and $(1-\theta_k^{-1})\vx^{k-1} + \theta_k^{-1} \vx^* \in Y$, thus
	\begin{align}
		\label{projection}
		\Vert (1-\theta_k^{-1})\vx^{k-1} + \theta_k^{-1} \vx^* - \vy^k\Vert_{\bD_k}^2 \leq \Vert \theta_k^{-1}(\vx^*-\textbf{v}^{k-1})\Vert_{\bD_k}^2,
	\end{align}
	moreover, $\textbf{v}^k = \vx^{k-1} +\theta_k(\vx^k-\vx^{k-1})$ implies that 
	\begin{equation}
		\label{convex_1}
		\Vert (1-\theta_k^{-1})\vx^{k-1}+\theta_k^{-1} 	\vx^*-\vx^k\Vert_{\bD_k}^2 = \Vert \theta_k^{-1}(\vx^*-\textbf{v}^k)\Vert_{\bD_k}^2.	
	\end{equation}
	Thus, combining \eqref{convex_key}, \eqref{projection}, and \eqref{convex_1}, we have
	\begin{equation*}
		t_k \theta_k^2(\Phi(\vx^k)-\Phi^*) + \frac{1}{2}\Vert \vx^*-\textbf{v}^k\Vert_{\bD_k}^2 \leq t_k(\theta_k^2-\theta_k)(\Phi(\vx^{k-1})-\Phi^*) + \frac{1}{2}\Vert \vx^*-\textbf{v}^{k-1}\Vert_{\bD_k}^2.
	\end{equation*}
	Moreover, the update of $\{\beta_k\}$ (line 6 of Algorithm \ref{algo:ASFISTA}) implies that $\theta_k^2-\theta_k - t_{k-1}/t_k \theta_{k-1}^2 = 0$. Plunge this into the above inequality, we prove $(i)$. \\
	$(ii)$ For $k\geq 1$, define $a_k = t_k\theta_{k}^2(\Phi(\vx^k)-\Phi^*)+\frac{1}{2}\Vert \vx^*-\textbf{v}^{k}\Vert_{\bD_k}^2$. Thus,
	\begin{align*}
		a_{k+1} &= t_{k+1}\theta_{k+1}^2(\Phi(\vx^{k+1})-\Phi^*)+\frac{1}{2}\Vert \vx^*-\textbf{v}^{k+1}\Vert_{\bD_{k+1}}^2 \\
		&\leq t_k \theta_k^2(\Phi(\vx^k)-\Phi^*) + \frac{1}{2}\Vert \vx^*-\textbf{v}^k\Vert_{\bD_{k+1}}^2 \\
		&\leq t_k \theta_k^2 (\Phi(\textbf{x}^k)-\Phi^*) + \frac{1+\eta_k}{2}\Vert \vx^*-\textbf{v}^k\Vert_{\bD_k}^2 \\
		&\leq (1+\eta_k)(t_k\theta_k^2(\Phi(\vx^k)-\Phi^*)+\frac{1}{2}\Vert \vx^*-\textbf{v}^k\Vert_{\bD_k}^2)\\
		&= (1+\eta_k)a_k,
	\end{align*} 
	where the first inequality follows from the targeted inequality in term $(i)$.  
	Then, invoking Lemma \ref{ine1} we conclude that the sequence $\{a_k\}$ converges and thus bounded. Let $C_1>0$ satisfy $a_k \leq C_1$ 
	for all $k$, then we have for $k\geq 1$
	\[
	\Phi(\vx^k)-\Phi^* \leq \frac{C_1}{t_k\theta_k^2}.
	\]
	Now, for $k \geq 2$ we have 
	\begin{equation*}
		\sqrt{t_{k-1}}\theta_{k-1} = \sqrt{\theta_k(\theta_k-1) t_k} 
		\leq  \sqrt{t_k}\theta_k - \frac{\sqrt{t_k}}{2}.
	\end{equation*}
	Summing the above inequality from 2 to $k$ and then by rearranging, we have $\sqrt{t_k}\theta_k\geq \sqrt{t_1} + \frac{1}{2}\sum_{i=2}^{k}\sqrt{t_k}$. Note $t_{\text{min}}:=\inf\{t_k\} >0$, thus for $k\geq 2$, 
	\[\frac{1}{t_k\theta_k^2} \leq \frac{1}{(\sqrt{t_1}+\frac{1}{2}\sum_{i=2}^{k}\sqrt{t_k})^2} \leq \frac{4}{(k+1)^2 t_{\text{min}}}.\] 
	Then, let $C = 4C_1/t_{\text{min}}$, it holds that for all $k\geq 2$
	\[
	\Phi(\vx^k)-\Phi^* \leq \frac{C}{(k+1)^2}.
	\]
	Therefore, we obtain the optimal convergence rate.
	\qed
\end{proof}

\section{Numerical results}\label{sec:5}
In this section, we conduct numerical simulations on problems of sparse binary logistic regression and compressed sensing with Poisson noise. All experiments are implemented in MATLAB 2019a on a 64-bit PC with an Intel(R) Core(TM) i5-6200U CPU (2.30GHz) and 8GB of RAM. 

\subsection{Sparse binary logistic regression}
In this part, we consider the following DC regularized sparse binary logistic regression model:  
\begin{equation}
	\label{ex1}
	\min_{\vx\in \R^n} \{\frac{1}{m}\sum_{i=1}^{m} \log (1+\exp(-b_i(\textbf{a}_i^T\vx)))+ \lambda \Vert \vx \Vert_1 - \lambda \Vert \vx \Vert_2\},
\end{equation}
where $\lambda >0$ is a regularization parameter and $\{(\textbf{a}_1,b_1),\cdots,(\textbf{a}_m,b_m)\}$ is a training set with observations $\textbf{a}_i\in\R^n$ and labels $b_i \in \{-1,1\}$. This fits \eqref{P} with $f(\vx) = \sum_{i=1}^{m} \log (1+\exp(-b_i(\textbf{a}_i^T\vx)))$, $g(\vx) = \lambda \Vert \vx \Vert_1$, and $h(\vx) = \lambda \Vert \vx\Vert_2$. Here, we use $\ell_{1-2}$ regularizer \cite{Jackl12} for promoting sparsity; see \cite{reweightedl1} for alternative nonconvex regularizers. For this model, $Y$ in Algorithm \ref{algo:ASpDCAe} is the whole space and $ {\rm \Pi}_Y^{\textbf{D}_k}$ is the identity operator. Besides, the computation about the proximal map $\text{Prox}_{t_k g}^{\textbf{D}_k}$ has closed form, which can be found in \cite[Example 6.8]{Beck_1order}.

Two tested datasets, \texttt{w8a} and \texttt{CINA}, obtained from \texttt{libsvm} \cite{libsvm} are utilized to show the performances of our proposed method SPDCAe (Algorithm \ref{algo:ASpDCAe}).  We compare four algorithms for \eqref{ex1}: the first two are versions derived from SPDCAe, the other are pDCAe \cite{wen2018proximal} and ADCA \cite{nhat2018accelerated}. Note that all of these algorithms are equipped with Nesterov's extrapolation, and here wo do not compare that out of this type, such as pDCA \cite{DC_Takeda} and the
general iterative shrinkage and thresholding algorithm (GIST) \cite{pmlr-v28-gong13a}, since in general these methods do not work better than that with extrapolation; see \cite{wen2018proximal} for a comparison of pDCA, pDCAe, and GIST. Next, we discuss the implementation details of the involved algorithms below:

\begin{itemize}
	\item[$\bullet$] \textbf{SPDCAe1} is SPDCAe with scaling and non-monotone backtracking line search. The scaling matrix $\bD_k$ is selected by the following procedure.
	Let $\textbf{g}^k = \nabla f(\vy^k)$ and $\bar{\textbf{g}}^k = \sum_{i=1}^{k} \textbf{g}^k \odot \textbf{g}^k$, then set $$\bD_k = \text{diag}\left(\max\left(\frac{1}{\gamma_k},\min\left(\gamma_k,\sqrt{\bar{\textbf{g}}^k + \varepsilon \textbf{e}_n}\right)\right)\right),$$ where $\varepsilon = 10^{-6}$ and $\gamma_k = \sqrt{1+\frac{10^{13}}{(k+1)^2}}$. Here, all of the above operators are conducted componentwise. Note that the update of the scaling matrices as above is motivated by that in the adaptive gradient (AdaGrad) algorithm \cite{AdaGrad}. The selection of $\{\gamma_k\}$ can guarantee that $\{\textbf{D}_k\}$ satisfies Assumption \ref{assumDk} since $\{\textbf{D}_k\}$ converges to the identity matrix. In addition, $\underline{L}$ is set as $10^{-10}$.  The factor $\eta$ is set as 2 and the initial guesses $L_1$ for the datasets are fixed as 1; then for $k\geq 2$, the initial guess $L_k$ is selected as $L_{k-1}/2$ if $k$ is not divisible by 5, otherwise as $L_{k-1}$.
	The selection of extrapolation parameter $\{\beta_k\}$ uses the fixed ($T_2 = 200$) and adaptive restart scheme  as introduced in Sect. \ref{sec:3}.
	\item[$\bullet$]  \textbf{PDCAe1} is SPDCAe without scaling and with non-monotone backtracking line search. Compared with that implementation of SPDCAe1, two modifications are in order. First, the initial guesses $L_1$ for the datasets are set as $0.1$. Second, the scaling matrix is the identity matrix. 
	\item[$\bullet$] \textbf{pDCAe} also uses the fixed $(T_2=200)$ and restart scheme for updating $\{\beta_k\}$. The smoothness parameter of logistic loss $f$ is estimated by computing the largest eigenvalue of the Hessian matrix of $f$.
	\item[$\bullet$] \textbf{ADCA} uses the same bound of smoothness parameter of $f$ as that in pDCAe. Besides, the parameter $q$ in ADCA is set as 3. 
\end{itemize}
  
In our experiments, $\lambda$ are all set as $10^{-3}$. The initial point for all of the involved algorithms is randomly generated by the MATLAB command $\texttt{rand(n,1)}$. We stop the algorithms via the \emph{relative error} $(F(\vx^k) - F^*)/{F^*} \leq tol$, where $\vx^*$ is an approximately solution derived by running PDCAe1 for 10000 iterations and $F^*$ is its corresponding function value.
 
Table \ref{tab:1} demonstrates the performances  for \texttt{w8a} and \texttt{CINA} datasets. The total number of iterations and CPU time (in seconds) averaged on 10 runs for $tol \in \{10^{-2i}:i\in\{1,2,3,4\}\}$ are reported. The mark ``Max'' means the number of iterations exceeds 10000. We visualize  in Fig. \ref{fig:logstic} the trend of the relative error with respect to the CPU time. It is observed tat SPDCAe1 performs the best among the involved algorithms, hitting all of the tolerances with the least amount of time and the least number of iterations. On the contrary, pDCAe does not work well, especially for \texttt{CINA} dataset. For \texttt{w8a} dataset, PDCAe1 is better than ADCA; while for \texttt{CINA} dataset the result is opposite. It seems that ADCA  converges very fast when approaching the tail. Next, by comparing SPDCAe1 with PDCAe1, the benefit from scaling is indicated; while by comparing PDCAe1 and pDCAe, we observe that the non-monotone backtracking promotes the performance. Finally, consider that SPDCAe1 performs the best, the benefits from both scaling and the non-monotone backtracking are indicated.

\begin{table}[h!]
	\caption{Total number of iterations and CPU time (in seconds) averaged on 10 runs for \texttt{w8a} and \texttt{CINA} datasets. Bold values correspond to the best results for each dataset}
	\label{tab:1}
	\centering
	
	\begin{tabular}{cccccccccc} 
		\hline
		\multirow{2}{*}{Dataset} & \multirow{2}{*}{Algorithm} &  \multicolumn{2}{c}{$tol: 10^{-2}$} & \multicolumn{2}{c}{$tol: 10^{-4}$} & \multicolumn{2}{c}{$tol: 10^{-6}$} & \multicolumn{2}{c}{$tol: 10^{-8}$} \\
		\cline{3-10}
		& & Time &Iter.  &Time & Iter. &Time &Iter. &Time & Iter. \\
		\hline
		\multirow{4}{*}{\texttt{w8a}}	& SPDCAe1  & \textbf{0.11} & \textbf{18} & \textbf{0.18} & \textbf{32}& \textbf{0.23} & \textbf{41} &  \textbf{0.29} & \textbf{50} \\
		& PDCAe1  & 0.20 & 36 & 0.32 & 57 & 0.38 & 70 & 0.48 & 89\\
		& pDCAe  & 0.84 & 148 & 2.93 & 587 & 5.84& 1113& 7.75& 1571\\
		& ADCA  & 1.03 & 153 & 2.12 & 356 & 3.08& 486& 4.24& 739 \\
		\hline
		\multirow{4}{*}{\texttt{CINA}} & SPDCAe1  & \textbf{0.16} & \textbf{186}& \textbf{0.49} & \textbf{684} & \textbf{0.72} & \textbf{1059}& \textbf{1.01} & \textbf{1524}\\
		& PDCAe1 & 0.47 & 743 & 1.42 & 1992& 3.56& 3799 & 3.89& 5964 \\
		& pDCAe & 3.43 & 5781&Max & Max& Max& Max& Max& Max \\
		& ADCA & 0.91 & 1082 & 1.29 & 1656 & 1.76 & 2467 & 2.28 & 3152\\
		\hline
	\end{tabular}
\end{table}

	
\begin{figure}[ht!]
	\subfigure[\texttt{w8a}]{\includegraphics[width=0.52\textwidth]{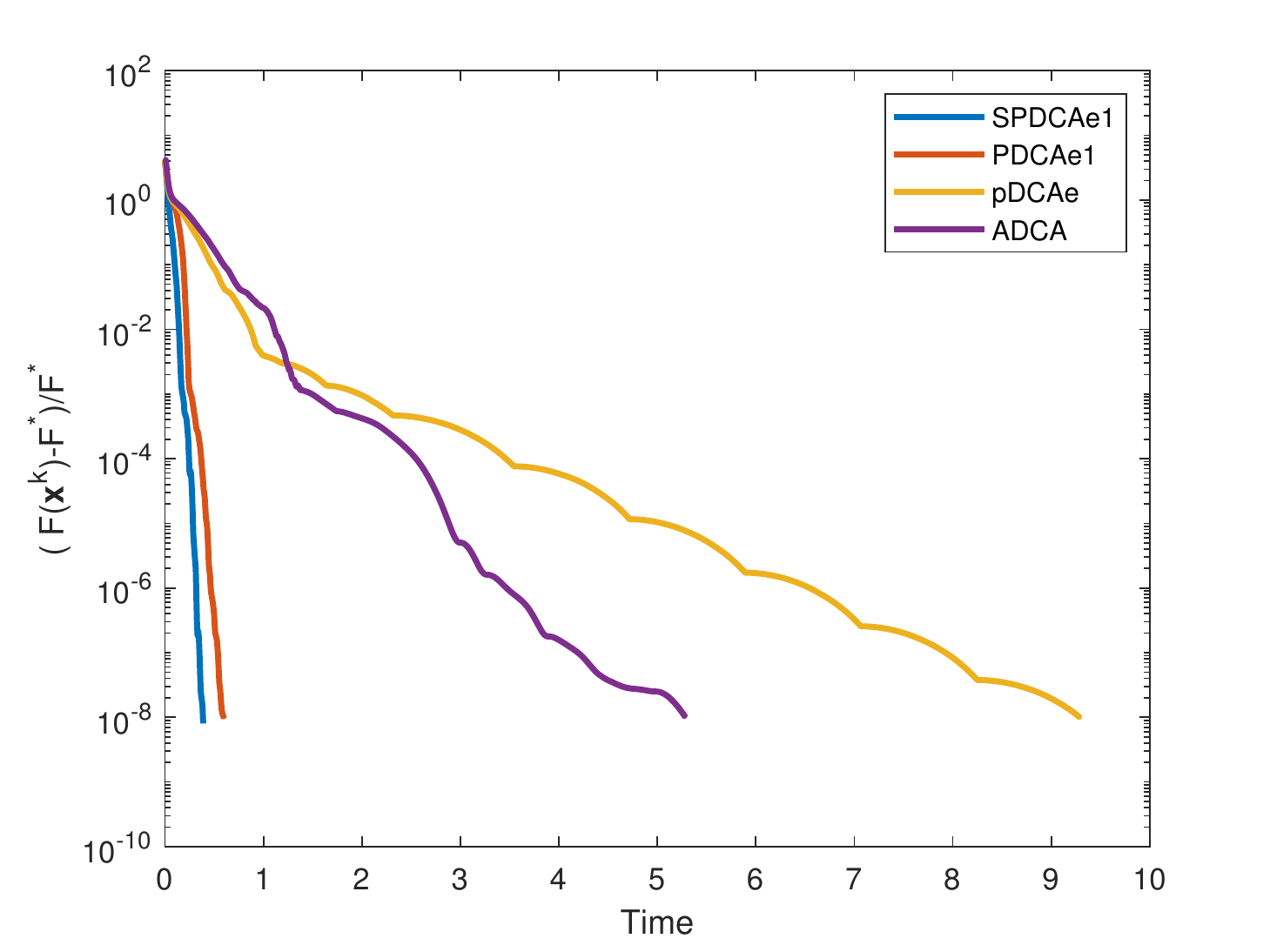}}
	\subfigure[\texttt{CINA}]{\includegraphics[width=0.52\textwidth]{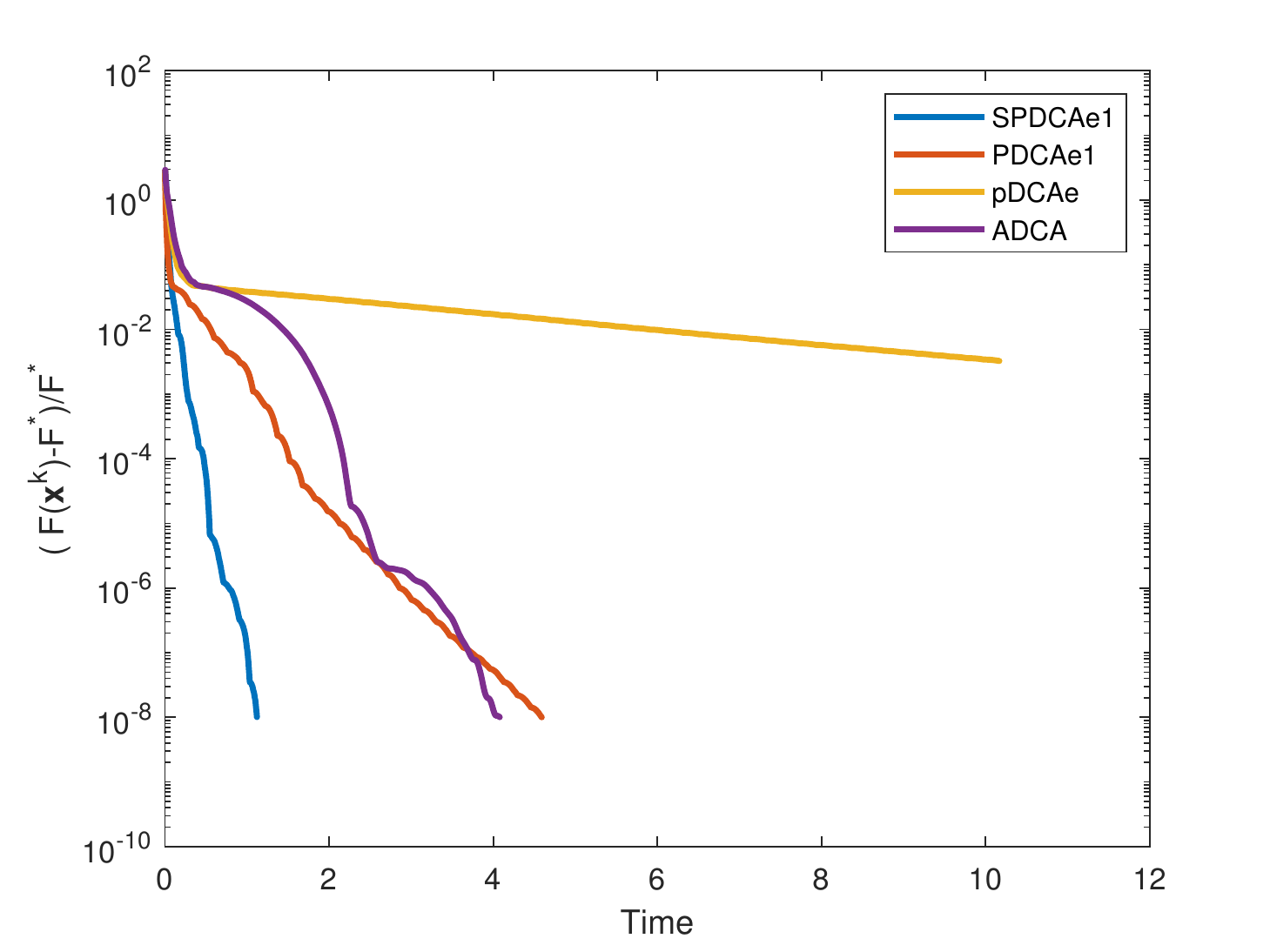}}
	\caption{Trend of the relative error for \texttt{w8a} and \texttt{CINA} datasets}
	\label{fig:logstic}      
\end{figure}

\subsection{Compressed sensing with Poisson noise}
In this subsection, we consider the problem of recovering  the sparse signal corrupted by Poisson noise. Specifically, we assume the observed data $\textbf{b}\in \R^m$ is the realization of a Poisson random vector with expected values being $\textbf{A}\vx_{\text{true}} + bg$, where  $\vx_{\text{true}}$ is the (actually unknown) signal of interest, $\textbf{A}\in\R^{m\times n}$ is the measurement matrix, and $bg$ (a small scalar close to 0) is a positive background. In case of Poisson noise, the generalized Kullback-Leibler divergence 
\begin{equation*}
	\text{KL}(\vx):= \sum_{i=1}^{n} \left\{b_i\log\frac{b_i}{(\textbf{A}\vx+bg)_i}+(\textbf{A}\vx+bg)_i-b_i \right\}
\end{equation*}  
is used to measure the distance of $\vx$ from the observed data $\textbf{b}$. To recover the true solution $\vx_{\text{true}}$, we use the following DC optimization model
\begin{equation}
	\label{ex2}
	\min_{\vx\in \R_{+}^n} \{\text{KL}(\vx) + \lambda \Vert \vx \Vert_1 - \lambda \Vert \vx \Vert_2\},
\end{equation}
where $\lambda >0$ is a regularization parameter. In \cite{SFBEM}, the $\ell_1$ norm is utilized for inducing sparsity, here we use $\ell_{1-2}$ penalty instead. Then \eqref{ex2} matches \eqref{P} with $f(\vx) = \text{KL}(\vx)$, $g(\vx) = \lambda \Vert \vx \Vert_1 + \delta_{\R_{+}}(\vx)$, and $h(\vx) = \lambda \Vert \vx\Vert_2$. Thus $Y$ in Algorithm \ref{algo:ASpDCAe} is the nonnegative orthant and $ {\rm \Pi}_Y^{\textbf{D}_k}$ is not up to the underling inner product. Besides, the computation about the proximal map $\text{Prox}_{t_k g}^{\textbf{D}_k}$ has closed form, which can be found in \cite[Lemma 6.5]{Beck_1order}.

We use the same procedure in \cite{SFBEM} to generate the experimental data. The measurement matrix $\textbf{A}$  actually depends on a probability, which is set as 0.9 in our test. For reader's convenience,  we describe the procedure below:

\begin{itemize}
	\item The matrix $\textbf{A}$ has been generated as detailed in \cite{CS_performance} so that $\textbf{A}$ preserves both the positivity and the flux of any signal (i.e., if $\textbf{z}\geq 0$, then $(\textbf{Az})_i \leq \sum_{i=1}^n z_i)$. 
	\item The signal $\vx_{\text{true}}\in\R^{5000}$ has all zeros except for 20 nonzeros entries drawn uniformly in the interval  $[0,10^5]$.
	\item The observed signal $\textbf{b}\in\R^{1000}$ has been obtained by corrupting the vector $\textbf{A}\vx_{\text{true}} + bg  \;  (bg=10^{-10})$ by means of the MATLAB \texttt{imnoise} function.
\end{itemize}  
 
Note that in this situation, the estimate of the bound of the smoothness parameter of KL is problematic since it is an extremely large number \cite{SPIRAL}. Thus, taking into account the efficiency, the aforementioned pDCAe and ADCA are not applicable; however, SPDCAe is still valid, from which we derive four versions: SPDCAe1, PDCAe1, SPDCAe0, and PDCAe0. The first two have already been introduced in the above subsection. Here, SPDCAe0 and PDCAe0 are their counterparts with monotone backtracking line search. Specifically, the first is SPDCAe with scaling and monotone backtracking, while the second is SPDCAe without scaling and with monotone backtracking. Our intention here is to show the effect from incorporating both the variable metric method and the non-monotone backtracking.

In our experiment, $\lambda$ is set as $10^{-3}$. The scaling matrices for SPDCAe1 and SPDCAe0 have been selected by writing the gradient of KL as
\[
-\nabla \text{KL}(\vx) = U_{\text{KL}}(\vx) - V_{\text{KL}}(\vx)
\]
with $U_{\text{KL}}(\vx)\geq 0$ and $V_{\text{KL}}(\vx)>0$; see \cite{LANTERI2001945} for a detailed description of such a decomposition. Then $\bD_k$ is defined as $$\bD_k = \text{diag}\left(\max\left(\frac{1}{\gamma_k},\min\left(\gamma_k, \frac{\vy^k}{V_{\text{KL}}(\vy^k)}  \right)\right)\right)^{-1}$$ with $\gamma_k = \sqrt{1+\frac{10^{13}}{(k+1)^2}}$.  Next for SPDCAe1 and PDCAe1, $\underline{L}$ is set as $10^{-10}$; besides, the factor $\eta$ in SPDCAe1 and PDCAe1 is set as 2, while that in SPDCAe0 and PDCAe0 is set as 1.2. The initial guesses $L_1$ for SPDCAe1 and SPDCAe0 are set as 0.1, and that for PDCAe1 and PDCAe0 are set as $10^{-5}$. Moreover, for $k\geq 2$, the initial guesses for SPDCAe1 and PDCAe1 are selected as $L_{k-1}/2$ if $k$ is not divisible by 5, otherwise as $L_{k-1}$. Next for the involved algorithms, the initial point is the $n$-dimensional column vector of all ones and the extrapolation parameter $\{\beta_k\}$ is all updated via the fixed ($T_2 = 200$) and adaptive restart scheme.  We stop the algorithms via the \emph{relative error} $(F(\vx^k) - F^*)/{F^*} \leq tol$, where $\vx^*$ is an approximately solution derived by running PDCAe1 for 10000 iterations and $F^*$ is its corresponding function value.

In Table \ref{tab:3}, we present in detail the total number of iterations and CPU time (in seconds) averaged on 10 runs for $tol \in \{10^{-i}:i\in\{1,2,\cdots,5\}\}$. The mark ``Max'' means the number of iterations exceeds 10000. Here, we observes that among all of the involved algorithms, PDCAe0 performs the worst, which requires 2910 steps  to reach the tolerance $tol = 10^{-1}$ and always hits the maximum iterations 10000 for other tolerances; on the contrary, SPDCAe1 performs the best, hitting all of the tolerances with the least amount of time and the least number of iterations. Next in Fig. \ref{subfig1:denoise}, the trend of the relative error with respect to the CPU time is plotted. Here, by comparing the trends of SPDCAe1 and SPDCAe0 (also PDCAe1 and PDCAe0) we observe the benefit from introducing the non-monotone backtracking; next by comparing SPDCAe1 and PDCAe1 (also SPDCAe0 and PDCAe0), the benefit from scaling is indicated; besides, by comparing SPDCAe0 and PDCAe1 we observe that the former benefits much from the scaling procedure and later is worse than PDCAe1. Anyway,  SPDCAe1 works the best, indicating the benefit from both the scaling method and the non-monotone backtracking. Finally, we observe from Fig. \ref{subfig2:denoise} that the solution recovered by SPDCAe1 is very close to 
the true solution.

As a conclusion, for the Poisson denoising problem, SPDCAe1 could be a promising algorithm since the benefit from incorporating the inexpensive diagonally scaling procedure for better local approximation and using the non-monotone backtracking for adaptive step-sizes selection.  

\begin{table}[h!]
	\caption{Total number of iterations and CPU time (in seconds) averaged on 10 runs. Bold values correspond to the best results}
	\label{tab:3}
	\centering

	\begin{tabular}{ccccccccccc} 
	\hline
	\multirow{2}{*}{Algorithm} & \multicolumn{2}{c}{$tol: 10^{-1}$} & \multicolumn{2}{c}{$tol: 10^{-2}$} & \multicolumn{2}{c}{$tol: 10^{-3}$} & \multicolumn{2}{c}{$tol: 10^{-4}$} & \multicolumn{2}{c}{$tol: 10^{-5}$} \\
	\cline{2-11}
	& Time &Iter. & Time & Iter. &Time & Iter. &Time &Iter. &Time & Iter. \\
	\hline
	SPDCAe1 &  \textbf{0.22} & \textbf{28} & \textbf{0.29} & \textbf{38} & \textbf{0.31} & \textbf{42} & \textbf{0.39} & \textbf{54} & \textbf{0.40} & \textbf{55} \\
	PDCAe1 &  0.57 &  78  & 0.80 & 101 & 0.81& 104 & 0.96 & 110 & 0.97  & 112 \\
	SPDCAe0 &  0.56 & 79& 1.86 & 298 &3.25 & 456 & 6.91 & 953 & 18.62 & 2508\\
	PDCAe0 &  19.66 & 2910& Max & Max & Max & Max & Max & Max & Max & Max \\
	\hline
	\end{tabular}
\end{table}

\begin{figure}[ht!]
  \subfigure[Trend of relative error]{\includegraphics[width=0.52\textwidth]{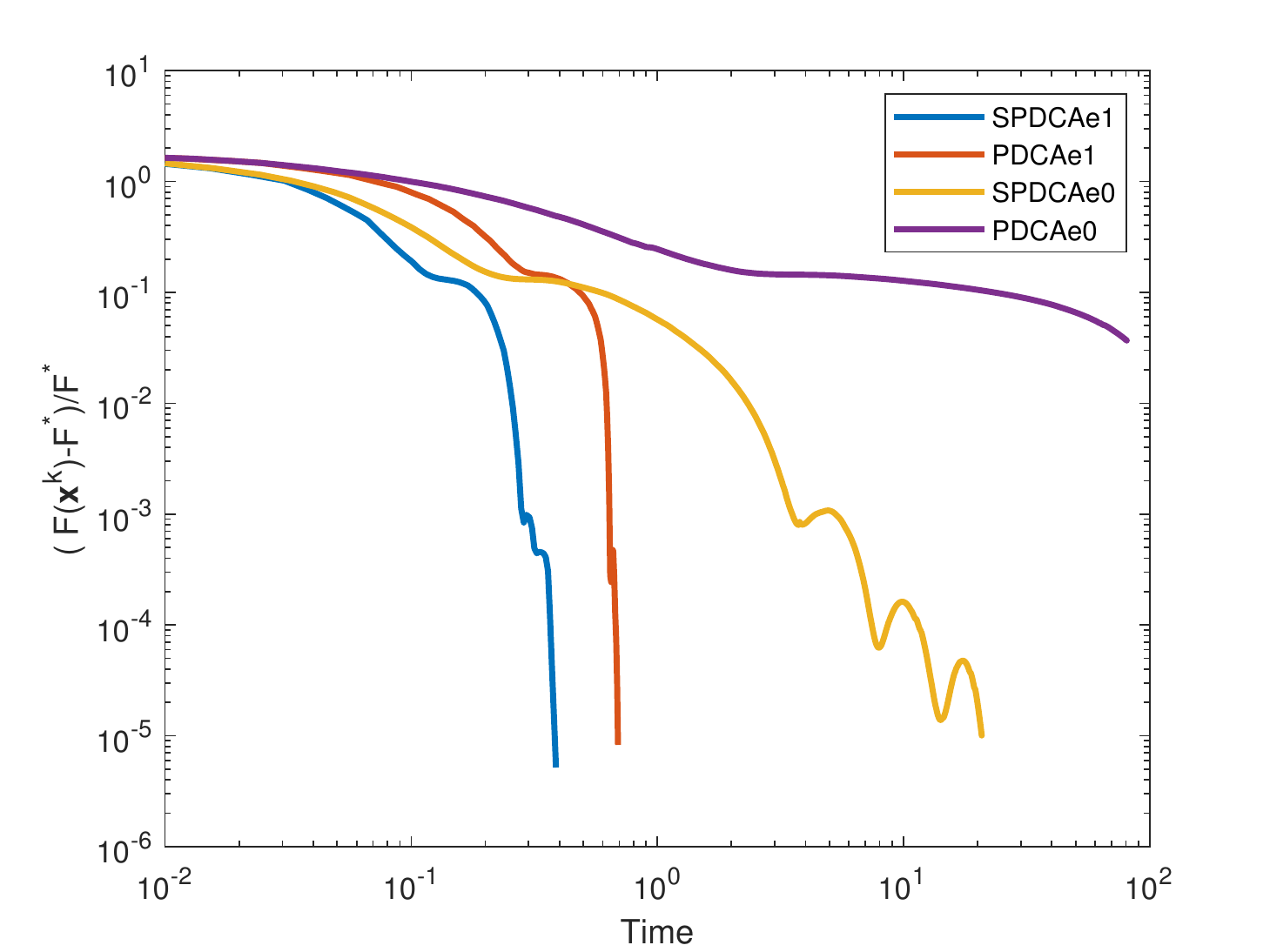}
  \label{subfig1:denoise}
	}
  \subfigure[The true solution and the solution obtained by SPDCAe1 under $tol = 10^{-5}$ ]{\includegraphics[width=0.52\textwidth]{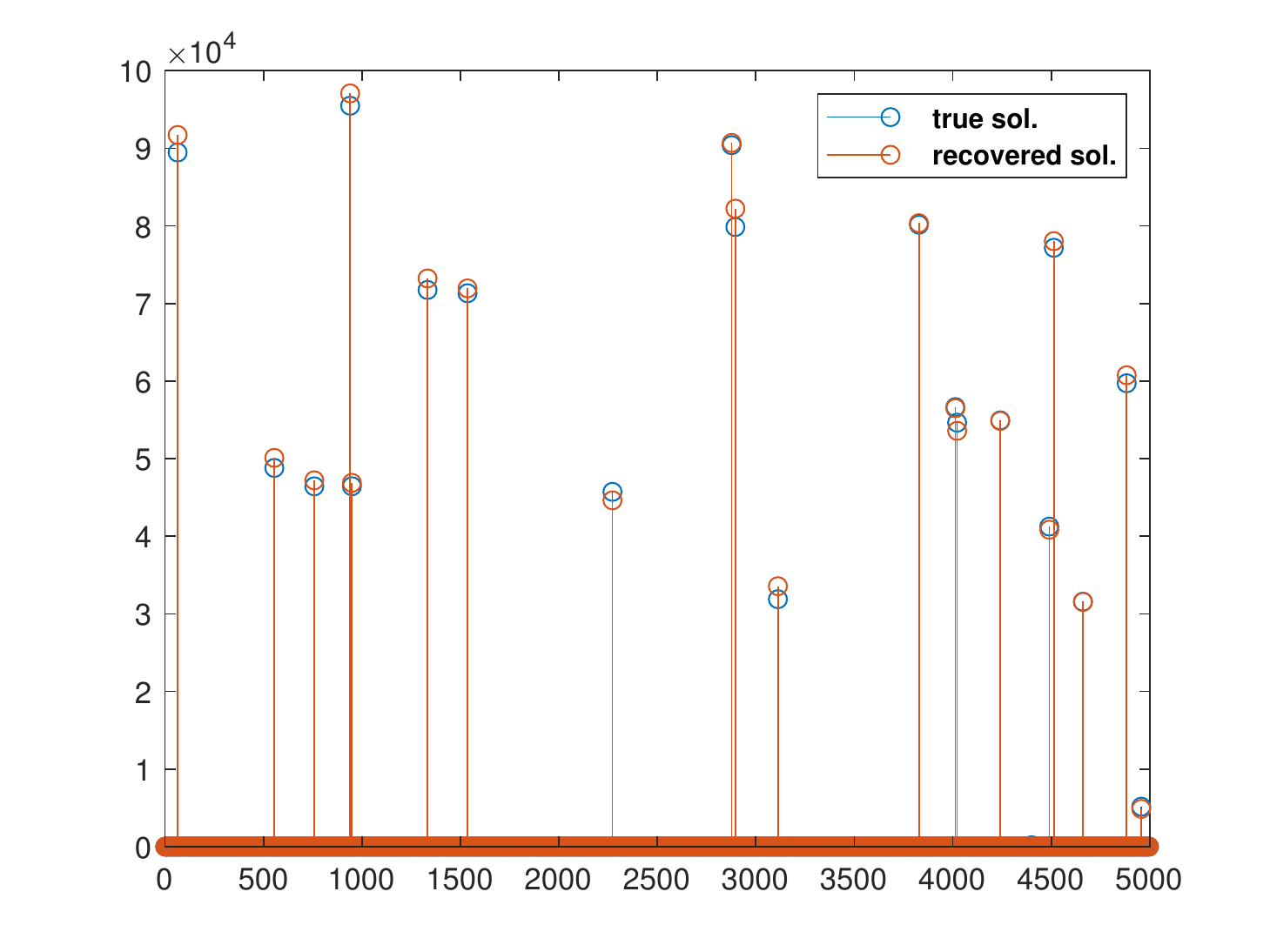}
  \label{subfig2:denoise}
	}
  \caption{Performance for the Poisson denoising problem}
\label{fig:denoise}      
\end{figure}

\section{Conclusion}
\label{sec:6}
In this paper, we propose a DC programming algorithm, called SPDCAe, for solving a composite DC program \eqref{P}, which incorporates the variable metric method, the Nesterov's extrapolation, and the backtracking line search (not necessarily monotone). We establish the subsequential convergence to a critical point of SPDCAe under suitable selections of the extrapolation parameters and the scaling matrices. Besides, we also demonstrate that the convex version of SPDCAe for \eqref{CP}, denoted by SFISTA, enjoys the optimal $\mathcal{O}(1/k^2)$ convergence rate in function values. This rate of convergence coincides with that of the well-known FISTA algorithm \cite{FISTA} and SFBEM \cite{SFBEM} (a scaling version of FISTA with monotone backtracking). Numerical simulations on sparse binary logistic regression demonstrate the good performances of our methods 
(especially that with scaling and non-monotone backtracking) compared with pDCAe \cite{wen2018proximal} and ADCA \cite{nhat2018accelerated}. Furthermore, for compressed sensing with Poisson noise problem, both pDCAe  and ADCA are not applicable, while our algorithm is still valid, where we show the benefits of including scaling and non-monotone backtracking.   

As further researches, first, the global sequential convergence of SPDCAe to a critical point under the  Kurdyka-Łojasiewicz  property \cite{bolte2007lojasiewicz,attouch2009convergence,attouch2013convergence} is worth noting; second, the SPDCAe could be compared with other accelerated algorithms, e.g., boosted-DCA \cite{aragon2018boosted,niu2019higher}, inertial-DCA \cite{de2019inertial,RInDCA} and accelerated methods based on second-order ODE \cite{niu2019discrete,francca2021gradient} for some suitable applications; moreover, it deserves developing some ingenious procedures for both scaling and adaptive step-sizes selection. Finally, it might be meaningful for designing some inexact variants for addressing the case where the computation of the proximal mapping of $g$ can not be exactly conducted.

\begin{acknowledgements} 
	This work is supported by the National Natural Science Foundation of China (Grant 11601327).
\end{acknowledgements}


\bibliographystyle{spmpsci}      
\bibliography{mybib}   

\begin{thebibliography}{10}
\providecommand{\url}[1]{{#1}}
\providecommand{\urlprefix}{URL }
\expandafter\ifx\csname urlstyle\endcsname\relax
  \providecommand{\doi}[1]{DOI~\discretionary{}{}{}#1}\else
  \providecommand{\doi}{DOI~\discretionary{}{}{}\begingroup
  \urlstyle{rm}\Url}\fi

\bibitem{surfaceDC}
Aleksandrov, A.D.: On the surfaces representable as difference of convex
  functions.
\newblock Sibirskie Elektronnye Matematicheskie Izvestiia \textbf{9}, 360--376
  (2012)

\bibitem{aragon2018boosted}
Artacho, F.J.A., Vuong, P.T.: The boosted dc algorithm for nonsmooth functions.
\newblock SIAM J. Optim. \textbf{30}(1), 980--1006 (2020)

\bibitem{attouch2009convergence}
Attouch, H., Bolte, J.: On the convergence of the proximal algorithm for
  nonsmooth functions involving analytic features.
\newblock Mathematical Programming \textbf{116}(1), 5--16 (2009)

\bibitem{attouch2013convergence}
Attouch, H., Bolte, J., Svaiter, B.F.: Convergence of descent methods for
  semi-algebraic and tame problems: proximal algorithms, forward--backward
  splitting, and regularized gauss--seidel methods.
\newblock Mathematical Programming \textbf{137}(1), 91--129 (2013)

\bibitem{Beck_1order}
Beck, A.: First-order methods in optimization.
\newblock SIAM, Philadelphia (2017)

\bibitem{FISTA}
Beck, A., Marc, T.: A fast iterative shrinkage-thresholding algorithm for
  linear inverse problems.
\newblock SIAM journal on imaging sciences \textbf{2}(1), 183--202 (2009)

\bibitem{beck2018globally}
Beck, A., Vaisbourd, Y.: Globally solving the trust region subproblem using
  simple first-order methods.
\newblock SIAM Journal on Optimization \textbf{28}(3), 1951--1967 (2018)

\bibitem{bertsekas2015convex}
Bertsekas, D.: Convex optimization algorithms.
\newblock Athena Scientific (2015)

\bibitem{bolte2007lojasiewicz}
Bolte, J., Daniilidis, A., Lewis, A.: The {\l}ojasiewicz inequality for
  nonsmooth subanalytic functions with applications to subgradient dynamical
  systems.
\newblock SIAM Journal on Optimization \textbf{17}(4), 1205--1223 (2007)

\bibitem{SFBEM}
Bonettini, S., Porta, F., Ruggiero, V.: A variable metric forward-backward
  method with extrapolation.
\newblock SIAM J. Optim. \textbf{38}(4), A2588--A2584 (2016)

\bibitem{BONETTINI2021113192}
Bonettini, S., Porta, F., Ruggiero, V., L., Z.: Variable metric techniques for
  forward-backward methods in imaging.
\newblock Journal of Computational and Applied Mathematics \textbf{385}, 113192
  (2021)

\bibitem{reweightedl1}
Cand{\`e}s, E.J., Wakin, M.B., Boyd, S.P.: Enhancing sparsity by reweighted
  $\ell_1$ minimization.
\newblock Journal of Fourier Analysis and Applications \textbf{14}, 877--905
  (2008)

\bibitem{libsvm}
Chang, C.C., Lin, C.J.: Libsvm: A library for support vector machines
  \textbf{2} (2011)

\bibitem{AdaGrad}
Duchi, J., Hazan, E., Singer, Y.: Adaptive subgradient methods for online
  learing and stochastic optimization.
\newblock J. Mach. Learn. Res \textbf{12}, 2121--2159 (2011)

\bibitem{francca2021gradient}
Fran{\c{c}}a, G., Robinson, D.P., Vidal, R.: Gradient flows and proximal
  splitting methods: A unified view on accelerated and stochastic optimization.
\newblock Physical Review E \textbf{103}(5), 053304 (2021)

\bibitem{pmlr-v28-gong13a}
Gong, P., Zhang, C., Lu, Z., Huang, J., Ye, J.: A general iterative shrinkage
  and thresholding algorithm for non-convex regularized optimization problems.
\newblock pp. 37--45. PMLR (2013)

\bibitem{DC_Takeda}
Gotoh J.~Y., T.A., Tono, K.: {DC formulations and algorithms for sparse
  optimization problems}.
\newblock Mathematical Programming \textbf{169}(1), 141--176 (2018)

\bibitem{SPIRAL}
Harmany, Z.T., Marcia, R.F., M, R.: This is spiral-tap: Sparse poisson
  intensity reconstruction algorithms--theory and practice.
\newblock In: IEEE Transactions on Image Processing, pp. 1084--1096 (2011)

\bibitem{hiriart1985generalized}
Hiriart-Urruty, J.B.: Generalized differentiability/duality and optimization
  for problems dealing with differences of convex functions.
\newblock In: Convexity and duality in optimization, pp. 37--70. Springer
  (1985)

\bibitem{Horst1999}
Horst, R., Thoai, N.V.: {DC} programming: Overview.
\newblock Journal of Optimization Theory and Applications \textbf{103}, 1--43
  (1999)

\bibitem{LANTERI2001945}
Lant{\'e}ri, H., Roche, M., Cuevas, O., Aime, C.: A general method to devise
  maximum-likelihood signal restoration multiplicative algorithms with
  non-negativity constraints.
\newblock Signal Processing \textbf{81}(5), 945--974 (2001)

\bibitem{le2001continuous}
Le~Thi, H., Pham~Dinh, T.: A continuous approach for large-scale constrained
  quadratic zero-one programming.
\newblock Optimization \textbf{45}(3), 1--28 (2001)

\bibitem{an2007new}
Le~Thi, H.A., Belghiti, M.T., Pham, D.T.: A new efficient algorithm based on dc
  programming and dca for clustering.
\newblock Journal of Global Optimization \textbf{37}(4), 593--608 (2007)

\bibitem{le2012dc}
Le~Thi, H.A., Moeini, M., Pham, D.T., Judice, J.: A dc programming approach for
  solving the symmetric eigenvalue complementarity problem.
\newblock Computational Optimization and Applications \textbf{51}(3),
  1097--1117 (2012)

\bibitem{an2003large}
Le~Thi, H.A., Pham, D.T.: Large-scale molecular optimization from distance
  matrices by a dc optimization approach.
\newblock SIAM Journal on Optimization \textbf{14}(1), 77--114 (2003)

\bibitem{Lethi_2005}
Le~Thi, H.A., Pham, D.T.: The dc (difference of convex functions) programming
  and dca revisited with dc models of real world nonconvex optimization
  problems.
\newblock Annals of operations research \textbf{133}(1-4), 23--46 (2005)

\bibitem{DCA30}
Le~Thi, H.A., Pham, D.T.: {DC programming and DCA: thirty years of
  developments}.
\newblock Mathematical Programming \textbf{169}(1), 5--68 (2018)

\bibitem{le2015dc}
Le~Thi, H.A., Pham, D.T., Le, H.M., Vo, X.T.: {DC approximation approaches for
  sparse optimization}.
\newblock European Journal of Operational Research \textbf{244}(1), 26--46
  (2015)

\bibitem{nesterov1983method}
Nesterov, Y.E.: A method for solving the convex programming problem with
  convergence rate $\mathcal{O} (1/{k^2})$.
\newblock In: Dokl. akad. nauk Sssr, vol. 269, pp. 543--547 (1983)

\bibitem{nesterov2013gradient}
Nesterov, Y.E.: Gradient methods for minimizing composite functions.
\newblock Mathematical Programming \textbf{140}(1), 125--161 (2013)

\bibitem{niu2010programmation}
Niu, Y.S.: Programmation dc et dca en optimisation combinatoire et optimisation
  polynomiale via les techniques de sdp.
\newblock Ph.D. thesis, INSA de Rouen, France (2010)

\bibitem{niu2019discrete}
Niu, Y.S., Glowinski, R.: Discrete dynamical system approaches for boolean
  polynomial optimization.
\newblock to appear in Journal of Scientific Computing, arXiv preprint
  arXiv:1912.10221  (2019)

\bibitem{niu2019improved}
Niu, Y.S., J{\'u}dice, J., Le~Thi, H.A., Pham, D.T.: Improved dc programming
  approaches for solving the quadratic eigenvalue complementarity problem.
\newblock Applied Mathematics and Computation \textbf{353}, 95--113 (2019)

\bibitem{niu2015solving}
Niu, Y.S., J{\'u}dice, J., Thi, H.A.L., Dinh, T.P.: Solving the quadratic
  eigenvalue complementarity problem by dc programming.
\newblock In: Modelling, Computation and Optimization in Information Systems
  and Management Sciences, pp. 203--214. Springer (2015)

\bibitem{niu2008dc}
Niu, Y.S., Pham, D.T.: A dc programming approach for mixed-integer linear
  programs.
\newblock In: International Conference on Modelling, Computation and
  Optimization in Information Systems and Management Sciences, pp. 244--253.
  Springer (2008)

\bibitem{niu2014dc}
Niu, Y.S., Pham, D.T.: Dc programming approaches for bmi and qmi feasibility
  problems.
\newblock In: Advanced Computational Methods for Knowledge Engineering, pp.
  37--63. Springer (2014)

\bibitem{niu2013efficient}
Niu, Y.S., Pham, D.T., Le~Thi, H.A., Judice, J.J.: Efficient dc programming
  approaches for the asymmetric eigenvalue complementarity problem.
\newblock Optimization Methods and Software \textbf{28}(4), 812--829 (2013)

\bibitem{niu2019higher}
Niu, Y.S., Wang, Y.J., Le~Thi, H.A., Pham, D.T.: Higher-order moment portfolio
  optimization via an accelerated difference-of-convex programming approach and
  sums-of-squares.
\newblock arXiv :1906.01509  (2019)

\bibitem{niu2019parallel}
Niu, Y.S., You, Y., Liu, W.Z.: Parallel dc cutting plane algorithms for mixed
  binary linear program.
\newblock In: World Congress on Global Optimization, pp. 330--340. Springer
  (2019)

\bibitem{niu2021difference}
Niu, Y.S., You, Y., Xu, W., Ding, W., Hu, J., Yao, S.: A difference-of-convex
  programming approach with parallel branch-and-bound for sentence compression
  via a hybrid extractive model.
\newblock Optimization Letters \textbf{15}(7), 2407--2432 (2021)

\bibitem{restartFISTA}
O'Donoghue, B., Cand{\`e}s, E.: Adaptive restart for accelerated gradient
  schemes.
\newblock Foundations of Computational Mathematics \textbf{15}, 715--732 (2015)

\bibitem{de2020abc}
de~Oliveira, W.: The abc of dc programming.
\newblock Set-Valued and Variational Analysis \textbf{28}(4), 679--706 (2020)

\bibitem{de2019inertial}
de~Oliveira, W., Tcheou, M.P.: An inertial algorithm for dc programming.
\newblock Set-Valued and Variational Analysis \textbf{27}(4), 895--919 (2019)

\bibitem{Pham_dca_theory}
Pham, D.T., Le~Thi, H.A.: Convex analysis approach to d.c. programming: theory,
  algorithms and applications.
\newblock Acta Math. Vietnam. \textbf{22}(1), 289--355 (1997)

\bibitem{Pham_trs}
Pham, D.T., Le~Thi, H.A.: A dc optimization algorithm for solving the
  trust-region subproblem.
\newblock SIAM Journal on Optimization \textbf{8}(2), 476--505 (1998)

\bibitem{phamdinh2014}
Pham, D.T., Le~Thi, H.A.: Recent advances in {DC} programming and {DCA}.
\newblock In: Transactions on {Computational} {Intelligence} {XIII}, pp. 1--37
  (2014)

\bibitem{pham2016dc}
Pham, D.T., Le~Thi, H.A., Pham, V.N., Niu, Y.S.: Dc programming approaches for
  discrete portfolio optimization under concave transaction costs.
\newblock Optimization letters \textbf{10}(2), 261--282 (2016)

\bibitem{pham2011efficient}
Pham, D.T., Niu, Y.S.: An efficient dc programming approach for portfolio
  decision with higher moments.
\newblock Computational Optimization and Applications \textbf{50}(3), 525--554
  (2011)

\bibitem{TAO1986249}
Pham, D.T., Souad, E.B.: Algorithms for solving a class of nonconvex
  optimization problems. methods of subgradients.
\newblock In: Fermat Days 85: Mathematics for Optimization, vol. 129, pp.
  249--271 (1986)

\bibitem{nhat2018accelerated}
Phan, D.N., Le, H.M., Le~Thi, H.A.: Accelerated difference of convex functions
  algorithm and its application to sparse binary logistic regression.
\newblock In: IJCAI, pp. 1369--1375 (2018)

\bibitem{CS_performance}
Raginsky, M., Willett, R.M., Harmany, Z.T., Marcia, R.F.: Compressed sensing
  performance bounds under poisson noise.
\newblock In: IEEE Transactions on Signal Process, vol.~58, pp. 3990--4002
  (2010)

\bibitem{rockafellar1970convex}
Rockafellar, R.T.: Convex analysis, vol.~36.
\newblock Princeton university press (1970)

\bibitem{AdaptiveFISTA}
Scheinberg, K., Goldfarb, D., Bai, X.: Fast first-order methods for composite
  convex optimization with backtracking.
\newblock Foundations of Computational Mathematics \textbf{14}, 389--417 (2014)

\bibitem{wen2018proximal}
Wen, B., Chen, X., Pong, T.K.: A proximal difference-of-convex algorithm with
  extrapolation.
\newblock Computational optimization and applications \textbf{69}(2), 297--324
  (2018)

\bibitem{Jackl12}
Yin, P., Lou, Y., He, Q., Xin, J.: Minimization of $\ell_{1-2}$ for compressed
  sensing.
\newblock SIAM J. Sci. Comput. \textbf{37}, A536--A563 (2016)

\bibitem{RInDCA}
You, Y., Niu, Y.S.: A refined inertial dc algorithm for dc programming.
\newblock Optimization and Engineering  (2022)

\end{thebibliography}

\end{document}